\documentclass[11pt]{article}
\usepackage{subfigure}
\usepackage{authblk}
\usepackage{hyperref}
\usepackage{float}
\usepackage{mathtools}

\usepackage{amssymb,amsthm,amsmath,mathrsfs,fullpage,multirow}
\usepackage{tikz}
\usetikzlibrary{matrix}

\makeatletter
\def\keywords{\xdef\@thefnmark{}\@footnotetext}
\makeatother

\renewcommand\S{\mathcal S}
\newcommand\A{\mathcal A}
\newcommand\cB{\mathcal B}
\newcommand\cZ{\mathcal Z}

\newcommand\C{\mathcal C}

\newtheorem{theorem}{Theorem}[section]
\newtheorem{lemma}[theorem]{Lemma}
\newtheorem{algorithm}[theorem]{Algorithm}
\newtheorem{proposition}[theorem]{Proposition}
\newtheorem{corollary}[theorem]{Corollary}
\newtheorem{conjecture}[theorem]{Conjecture}
\theoremstyle{definition}
\newtheorem{definition}[theorem]{Definition}
  \newtheorem{question}[theorem]{Question}
\newtheorem{remark}[theorem]{Remark}
\newtheorem{example}[theorem]{Example}

\usepackage{amsmath}

\newcommand{\thistheoremname}{}
\newtheorem*{genericthm*}{\thistheoremname}
\newenvironment{namedthm*}[1]
  {\renewcommand{\thistheoremname}{#1}%
   \begin{genericthm*}}
  {\end{genericthm*}}

\usepackage{enumerate}

\title{A lower bound on the growth rate of $(132,213)$-avoiding \\ cyclic permutations}
\author{Robert Laudone}
\affil{{\small Department of Mathematics, United States Naval Academy, Annapolis, MD, 21402}\\{\small Email: laudone@usna.edu }}

\date{}

\begin{document}


\maketitle

\begin{abstract}
    We construct a new reduction process which takes a $(132,213)$-avoiding permutation to a shorter one that is cyclic if and only if the original was. Iterating it determines whether a given $(132,213)$-avoiding permutation is cyclic. Reversing it gives four moves that build every cyclic $(132,213)$-avoiding permutation, uniquely, from $1$ if $n$ is odd, and $21$ if $n$ is even. Our main application is the first non-trivial lower bound for the growth rate of $\C_n(132,213)$, the cyclic permutations of length $n$ avoiding $132$ and $213$. We also give several other consequences of the reduction, including a bijection between the odd and even size classes and an exact enumeration for those permutations with a restricted number of layers.
\end{abstract}

\section{Introduction}

There is an interesting tension between avoidance in a permutation's one-line notation and how it decomposes into cycles. Consider $\pi = 4567312$. In one-line notation it is easily seen to be three decreasing blocks of increasing runs, so one can immediately tell it avoids $132$ and $213$, but we are largely blind to the cycle structure. Its cycle notation, $\pi = (1,4,7,2,5,3,6)$, shows at once that $\pi$ is cyclic, but it hides the avoidance. The two perspectives don't easily see one another. This has led to many interesting results and still open questions. For example, authors have studied how avoidance interacts with restricted cycle type \cite{AG22,AL25,BD20,DRS07}. In \cite{BT22}, Berman and Tenner introduced arrow patterns as a potential way to bridge this gap while studying shallow permutations. There have also been results enumerating classes of permutations that avoid a pattern in their one line notation, and whose cycle notation in some way avoids a possibly different pattern \cite{ABBGJ23, AL25-pa}.

The most famous question in this area was originally posed by Richard Stanley, who asked: how many cyclic permutations of length $n$ avoid a permutation $\sigma \in \S_3$? We denote these permutations by $\C_n(\sigma)$ and let $c_n(\sigma) = |\C_n(\sigma)|$. This problem has proved very challenging, as evidenced by the fact that all of the cases of $\C_n(\sigma)$ for $\sigma \in \S_3$ are still open.

Even understanding how these classes grow has been difficult. In \cite[Conjecture 5.2]{BC19}, B\'ona and Cory conjectured $2c_{n-1}(\sigma) \leq c_n(\sigma) \leq 4 c_{n-1}(\sigma)$ for $\sigma \in \S_3$. In \cite{AGL25}, we proved and sometimes improved the lower bounds of these conjectures for $\sigma \not= 123$ by showing cyclic permutations avoiding a given pattern admit a partial groupoid structure. In \cite{La26}, we showed that $\C_n(321)$ can be characterized purely by vincular avoidance in the image of Foata's fundamental bijection, which gave both structural results and growth bounds. Exact enumeration remains elusive in every one of these cases.

The case of cyclic permutations that avoid two patterns of size three has been more tractable. The number of cyclic permutations avoiding a consecutive $123$ or $321$ pattern was enumerated in \cite{ET19}, and the authors of \cite{AE14, BC19} enumerated all cyclic permutations avoiding a pair of patterns of length three, with a single exception: the pair $(132,213)$. For this remaining case, the only progress has been an upper bound of $n^2 \cdot 2^{n/2}$, giving $\sqrt{2} \approx 1.4142$ as an upper bound on the exponential growth rate, established in \cite[Theorem 1.4]{H19}. No non-trivial lower bound was known. Our main result supplies one,

\begin{namedthm*}{Theorem \ref{thm: lowerBound}}
    $\liminf_{n \to \infty} |\C_{n}(132,213)|^{1/n} \geq  1.23693$.
\end{namedthm*}

For comparison, Huang experimentally estimated the growth rate to be around $1.301$ \cite[Section 5]{H19}. The driving tool behind this paper is a new reduction process (Theorem \ref{thm: reduction}) that checks whether a given $\pi \in \S_n(132,213)$ is cyclic. Our reduction works for every $\pi \in \S_n(132,213)$, not just the balanced ones in the sense of \cite{H19} (see Section \ref{sec: background} for the definition of balanced). 

The reduction turns out to be useful well beyond the bound. Its inverse moves (Lemma \ref{lem: moves}) let us build every element of $\C_n(132,213)$ uniquely from the permutations $1$, if $n$ is odd, or $21$ if $n$ is even, which is what drives the lower bound. It recovers Huang's parity result and also produces a bijection between the odd and even size classes (Theorem \ref{thm: bijs-odd-even}). Finally, it yields exact enumerations for those $\pi \in \C_n(132,213)$ whose corresponding compositions have length two or three (Section \ref{sec: restricted-parts}).

The paper is structured as follows. In Section \ref{sec: background}, we set up notation and prove two structural lemmas. In Section \ref{sec: reduction}, we state and prove the reduction process, introduce the inverse moves and the signature of a permutation, and derive the consequences mentioned above. Then in Section \ref{sec: lower-bound}, we prove the lower bound. We also include a second approach which currently gives a weaker bound of $1.1748$ but which we think has room to improve. In Section \ref{sec: restricted-parts}, we give exact enumerations for permutations in $\C_n(132,213)$ whose corresponding compositions have restricted length, and in Section \ref{sec: balanced} we examine how the reduction interacts with Huang's balanced permutations. We conclude in Section \ref{sec: future-work} with some open questions and conjectures.

\section{Background and Notation} \label{sec: background}

Let $\S_n$ denote the symmetric group on $[n] = \{1,2,\dots,n\}$. We will focus on two main representations of a permutation. The first is its \textit{one-line notation}, $\pi = \pi_1\pi_2 \ldots \pi_n$ where $\pi_i = \pi(i)$. The second is its \emph{cycle notation}, expressing $\pi$ as a product of disjoint cycles. We say that a permutation is \emph{cyclic} if its cycle form consists of a single cycle. We denote the cyclic permutations of length $n$ by $\C_n$. We say that a permutation $\pi = \pi_1\cdots\pi_n$, written in one-line notation, \emph{contains} $\sigma = \sigma_1\sigma_2\ldots \sigma_k$ if there are indices $i_1 < i_2 < \cdots < i_k$ with $\pi_{i_r} < \pi_{i_s}$ if and only if $\sigma_r < \sigma_s$. In this case, we also say that $\pi_{i_1}\ldots \pi_{i_k}$ is order isomorphic to $\sigma$ and write ${\rm red}(\pi_{i_1}\ldots \pi_{i_k}) = \sigma$. If $\pi$ does not contain $\sigma$, we say $\pi$ \emph{avoids} $\sigma$. We denote the permutations of length $n$ avoiding a set of permutations $\{\sigma_1,\dots,\sigma_\ell\}$ by $\S_n(\sigma_1,\dots,\sigma_\ell)$, and cyclic permutations of length $n$ avoiding those patterns by $\C_n(\sigma_1,\dots,\sigma_\ell)$. 

The main focus of this paper will be $\C_n(132,213)$. $\S_n(132,213)$ is a special class of permutations, sometimes called \emph{reverse layered permutations}. This is because they have a very specific structure. To adequately describe this structure we need the following definitions. Given $\sigma = \sigma_1 \cdots \sigma_n \in \S_n$ and $\tau = \tau_1 \cdots \tau_m \in \S_m$, their \textit{direct sum} is $\sigma \oplus \tau = \sigma_1 \cdots \sigma_n (\tau_1 + n) \cdots (\tau_m + n)$ and their \textit{skew sum} is $\sigma \ominus \tau = (\sigma_1 + m) \cdots (\sigma_n + m) \tau_1\cdots \tau_m$. Furthermore, let $\rm{id}_n = 12\cdots n$ denote the identity permutation of length $n$. Permutations in $\S_n(132,213)$ are all of the form
\[
{\rm id}_{a_1} \ominus {\rm id}_{a_2} \ominus \cdots \ominus {\rm id}_{a_k}
\]
where $a_1 + a_2 + \cdots + a_k = n$. For example, our previous $\pi = 4567312$ is precisely ${\rm id}_{4} \ominus {\rm id}_{1} \ominus {\rm id}_{2}$. This makes it clear that $\pi \in \S_n(132,213)$ are precisely in bijection with compositions of size $n$. In this paper, a composition of size $n$ and length $k$ is a tuple of positive integers $(a_1,\dots,a_k)$ with $a_1+a_2+\cdots+a_k = n$.

\begin{definition}
    Given $\pi \in \S_n(132,213)$ with $\pi = {\rm id}_{a_1} \ominus {\rm id}_{a_2} \ominus \cdots \ominus {\rm id}_{a_k}$, we let $\kappa(\pi) = (a_1,a_2,\dots,a_k)$ be the composition corresponding to $\pi$. Similarly, $\kappa^{-1}(a_1,\dots,a_k)$ is the permutation $\pi \in \S_n(132,213)$ with $\kappa(\pi)=(a_1,\dots,a_k)$.
\end{definition}

We will often refer to the permutation and its composition interchangeably. Huang \cite[Section 2.2]{H19} defined a special subset of $\C_{2n}(132,213)$ which he called balanced. We will refer to them occasionally so we introduce the notation here. A permutation $\pi \in \C_{2n}(132,213)$ is \emph{balanced} if $\kappa(\pi) = (a_1,\dots,a_k)$ has some $1 \le j < k$ with $a_1 + a_2 + \cdots + a_j = n$. This means at some point the composition sums to precisely half of its size. We refer to this subset of balanced permutations as $\C_{2n}^B(132,213)$.

Our results will focus on the entries on the outer edges of the composition, so we develop some notation to make this easier to work with. Given $\pi \in \S_n(132,213)$ with $\kappa(\pi) = (a_1,a_2,\dots,a_{k-1},a_k)$, we call the elements on the outer boundary the \emph{active boundary condition}. We denote them as $a_1|a_k$ where $a_1$ is the left-most entry and $a_k$ is the right-most entry in the composition. If we write $a_1,a_2|$ we mean the composition begins with $a_1,a_2$. If we write $|a_{k-1},a_k$ we mean the composition ends with $a_{k-1},a_k$. We can think of the bar as representing all the remaining entries of the composition. 

To each composition, associate its set $S_c$ of partial sums. For example, the composition $(1,3,2,1)$ corresponds to $S_c = \{1,4,6\}$, we exclude $n$ from the set. Our ultimate goal is to understand which compositions of size $n$ correspond to $\pi \in \C_n(132,213)$. We now prove a couple structural results that will be useful in future sections,

\begin{lemma} \label{lem: balance}
    If $\pi \in \S_n(132,213)$ is cyclic, for $i \in [1,\dots,\lfloor (n-1)/2 \rfloor]$, we cannot have $i \in S_{\kappa(\pi)}$ and $n-i \in S_{\kappa(\pi)}$.
\end{lemma}

\begin{proof}
    If both $i$ and $n-i$ are in the set of partial sums, this means that the first $i$ and last $i$ elements of the permutation form a cycle. Indeed if $i$ is in the set of partial sums, the first $i$ elements of $\pi$ are $\{1,\dots,n-i\}$ and if $n-i$ is in the set of partial sums, the final $i$ elements are $\{1,\dots,i\}$. If $i \leq \lfloor (n-1)/2 \rfloor$, this means $\pi$ consists of at least two cycles.
\end{proof}

\begin{example}
    For example, consider the set of partial sums $\{2,3,5,6\}$ with $n = 7$. This corresponds to the composition $(2,1,2,1,1)$. If $i = 2$, we see that $2$ and $5$ are both in the set of partial sums. As a result, if we consider the corresponding permutation $\pi = 6753421$ the first two elements and final two elements form a cycle $(1,6,2,7)$.
\end{example}

We often only want to consider compositions $(a_1,\dots,a_k)$ where $a_1 > a_k$, the following result makes this possible,

\begin{proposition} \label{prop: reverse}
    For $n \geq 1$, $\kappa^{-1}(a_1,a_2,\dots,a_k) \in \C_n(132,213)$ if and only if $\kappa^{-1}(a_k,\dots,a_2,a_1) \in \C_n(132,213)$.
\end{proposition}

\begin{proof}
    If $\kappa(\pi) = (a_1,a_2,\dots,a_k)$, then $\kappa(\pi^{-1}) = (a_k,\dots,a_2,a_1)$. The inverse of a cyclic permutation is cyclic and $132$ and $213$ are both self inverse, so taking inverses preserves $132$ and $213$ avoidance.
\end{proof}

\section{Reduction Process} \label{sec: reduction}

In this section we prove a new reduction process to check when elements of $\S_n(132,213)$ are cyclic. For the following result, we will always assume that in $\kappa(\pi) = (a_1,a_2,\dots,a_k)$ we have $a_1 > a_k$. We can do this by Proposition \ref{prop: reverse} because one of $\kappa(\pi)$ or $\kappa(\pi^{-1})$ has this property and $\kappa(\pi) \in \C_n(132,213)$ if and only if $\kappa(\pi^{-1}) \in \C_n(132,213)$. First we need one more proposition,

\begin{proposition} \label{prop: unequal-ends}
    For $n > 2$, if $\pi \in \C_n(132,213)$ with $\kappa(\pi) = (a_1,\dots,a_k)$, then $a_1 \not= a_k$.
\end{proposition} 

\begin{proof}
    Suppose $a_1 = a_k = i$, then both $i$ and $n-i$ are in $S_{\kappa(\pi)}$, and Lemma \ref{lem: balance} rules this out whenever $i \le \lfloor (n-1)/2 \rfloor$. Otherwise $2i \ge n$, and since $i$ is a proper partial sum this forces $k = 2$ with $a_1 = a_2 = n/2$. But then $\pi = {\rm id}_{n/2} \ominus {\rm id}_{n/2}$, which is a product of $n/2$ transpositions and so isn't cyclic for $n > 2$.
\end{proof}

We are now ready to state our main reduction result,

\begin{theorem} \label{thm: reduction}
    For $n \geq 3$ and $k \ge 2$, $\pi \in \C_{n}(132,213)$ where $\kappa(\pi) = (a_1,a_2,\dots,a_k)$ if and only if $r(\pi) \in \C_{m}(132,213)$ where 
    \[
    \kappa(r(\pi)) =
    \begin{cases}
        (a_k,a_2,\dots,a_{k-1},2a_k-a_1), & 2a_k > a_1 > a_k\\
        (a_1 - 2a_k, a_{k}, a_2,\dots,a_{k-1}), & a_1 > 2a_k\\
        (a_k,a_2,\dots,a_{k-1}), & a_1 = 2a_k\\
    \end{cases}
    \]
    where in the first and last case $m = n-2(a_1-a_k)$, in the second case $m = n-2a_k$, so $m$ and $n$ have the same parity.
\end{theorem}

\begin{proof}
    Suppose $\kappa(\pi) = (a_1,a_2,\dots,a_k)$ with $a_1 > a_k$, which we may assume without loss by Proposition \ref{prop: reverse}. In each case below we delete two blocks of entries and splice the remaining cycle back together. The deleted blocks are determined by the resulting composition, so each surgery is reversible and the argument establishes both directions at once. For this reason we phrase everything in the forward direction, starting from $\pi \in \C_n(132,213)$.

    First, notice generally we can remove at most $\min(a_k,a_1-a_k)$ elements from the first and last increasing sequence in our permutation. Indeed, if $2a_k > a_1$, then $a_1 - a_k < a_k$, so we can only remove $a_1-a_k$ elements from the first increasing sequence before we reach position $a_k$. If $a_1 > 2a_k$, we cannot remove more than the entire final increasing sequence, so in this case we remove $a_k$ elements. In other words, the first and last case, which we will call case $A$ and case $C$, remove $a_1 - a_k$ elements from each end, while the second case, which we call case $B$, removes $a_k$ from each end.

    \medskip
    \noindent\textbf{Case $A$: $2a_k > a_1 > a_k$.} The beginning and end of the permutation look like
    \[
    \begin{gathered}
    \underset{1}{\underline{(n-a_1+1)}} \cdots \overbracket{\underset{2a_k - a_1 + 1}{\underline{(n-2(a_1-a_k) + 1)}}\cdots\underset{a_k}{\underline{(n-a_1+a_k)}}}^{B_1: \; a_1-a_k \; \text{elements}}\; \overbracket{\underset{a_k+1}{\underline{(n-a_1+a_k+1)}}\cdots \underset{a_1-1}{\underline{(n-1)}} \;\underset{a_1}{\underline{(n)}}}^{B_2: \; a_1-a_k \; \text{elements}} \\  |12 \cdots \underset{n-a_1+a_k}{\underline{(2a_k-a_1)}} \; \underbracket{\underset{n-a_1+a_k+1}{\underline{(2a_k-a_1+1)}}\cdots \underset{n}{\underline{(a_k)}}}_{B_3: \; a_1-a_k \; \text{elements}}
    \end{gathered}
    \]
    Notice this makes sense because $2a_k > a_1$, so that $a_1 - 2(a_1-a_k) = 2a_k - a_1 > 0$, and there are elements to the left of position $2a_k - a_1 + 1$.

    The elements of $B_2$ sit in positions $a_k+1,\dots,a_1$ and map to entries at the end of the permutation. Writing $j$ for whichever element maps to $a_1 - i$, we have
    \[
    j \to (a_1 - i) \to (n-i) \to (a_k-i) \to n- (a_1-a_k) -i
    \]
    for $0 \leq i \leq a_1 - a_k - 1$. In other words, each element of $B_2$ maps to something in $B_3$, which then maps back into $B_1$, whose elements are $n-2(a_1-a_k) + 1,\dots,n-a_1+a_k$. These in turn map to the elements immediately to the left of $B_3$.

    So when we remove all the elements of $B_2$ and $B_3$ and shift every remaining element larger than $a_k$ down by $a_1-a_k$, the resulting permutation is still cyclic. Indeed, we are only removing the $(a_1 - i) \to (n-i)$ portion of each chain and instead mapping $j$ directly to $n-a_1 + a_k - i$. Since we have removed $a_1-a_k$ elements from the first and last increasing sequences and left the interior parts alone, the resulting composition is $(a_k,a_2,\dots,a_{k-1},2a_k-a_1)$ as desired.

    \medskip
    \noindent\textbf{Case $C$: $a_1 = 2a_k$.} This is a degeneration of case $A$, so we treat it here. Now $a_1 - a_k = a_k$, so $B_2$ and $B_3$ together account for the entire final increasing run. Therefore, removing them deletes all of $\{1,\dots,a_k\}$ along with the top $a_k$ values. We still have a cyclic permutation so long as the original permutation was cyclic. The only difference here is that $B_3$ now exhausts the last part of the composition rather than truncating it. As a result, the composition loses its last part entirely and the result is $(a_k, a_2,\dots,a_{k-1})$, of size $n - 2a_k = n - 2(a_1-a_k)$.

    \medskip
    \noindent\textbf{Case $B$: $a_1 > 2a_k$.} Recall that we will be removing $a_k$ elements, but now there is a gap between the elements in positions $1,\dots,a_k$ and those in positions $a_1 - a_k + 1,\dots,a_1$. So now our permutation looks like
    \[
     \begin{gathered}
    \overbracket{\underset{1}{\underline{(n-a_1+1)}} \cdots \underset{a_k}{\underline{(n-a_1 + a_k)}}}^{B_1: \; a_k \; \text{elements}}  \overbracket{ \underset{a_k+1}{\underline{(n-a_1+a_k+1)}} \cdots\underset{a_1-a_k}{\underline{(n-a_k)}}}^{B_2: \; a_1-2a_k \; \text{elements}} \overbracket{\underset{a_1-a_k +1}{\underline{(n-a_k+1)}}\cdots \underset{a_1-1}{\underline{(n-1)}} \;\underset{a_1}{\underline{(n)}}}^{B_3: \; a_k \; \text{elements}} \\ 
     \; | \;\underset{n-a_k}{\underline{\pi_{n-a_k}}} \;\underbracket{\underset{n-a_k+1}{\underline{(1)}} \cdots \underset{n}{\underline{(a_k)}}}_{B_4: \; a_k \; \text{elements}}
    \end{gathered}
    \]
    Tracking the elements of the cycle we find the same chain $(a_1-i) \to (n-i) \to (a_k-i) \to n-a_1 + a_k - i$ and perform the same replacement. The only difference now is that the $a_1-2a_k$ elements of $B_2$ get skipped over. We remove $B_3$ and $B_4$, replace the elements of $B_3$ by those of $B_1$, and leave $B_2$ untouched. To compensate for removing $1,\dots,a_k$ we shift all the elements down by $a_k$. The resulting permutation is still cyclic.

    It remains to check that the result still avoids $132$ and $213$, since here the relative order has changed. Every element of $B_1$ is smaller than every element of $B_2$, and every element of $B_1$ is larger than everything to the right of $B_3$, the block it replaces. So the new permutation is still a skew sum of increasing runs, and it has the form $(a_1-2a_k,a_k,a_2,\dots,a_{k-1})$ as desired.
\end{proof}

\begin{remark}
    In the statement above we use the convention that a segment $a_2,\dots,a_{k-1}$ is empty when $k = 2$, and consists of the single entry $a_2$ when $k = 3$. Thus for $k = 2$ the first case reads $(a_2, 2a_2 - a_1)$, the second reads $(a_1 - 2a_2, a_2)$, and the third reads $(a_2)$.
\end{remark}

    In the rest of the paper, we continue to use $A$, $B$ and $C$ to refer to these three cases, and if $a_1 < a_k$ we call this case $R$, for reverse. The easiest way to understand this cycle surgery is with an example.
    
\begin{example}
     First, let's consider the permutation with composition $\kappa(\pi) = (6,1,4)$ this corresponds to the permutation $6789(10)(11)51234$ which is cyclic. Notice that $a_1 = 6$ and $a_3 = 4$. We want to remove the final $a_1 - a_3 = 2$ elements of the first and last increasing sequence, then shift all the elements larger than $4$ down by $2$ to compensate for removing $3$ and $4$. Considering the entire cycle, we will highlight the parts we will remove
    \[
    1 \to {\color{red} 6 \to 11} \to 4 \to 9 \to 2 \to 7 \to {\color{red} 5 \to 10} \to 3 \to 8 \to 1
    \]
    when we remove the red parts and shift everything larger than $4$ down by $2$ we find,
    \[
    1 \to 4 \to 7 \to 2 \to 5 \to 3 \to 6 \to 1
    \]
    which corresponds to the permutation $4567312$.
\end{example}

Now let's do an example to illustrate how this reduction process works at the composition level,

\begin{example}
    We'll do a few examples to see how this works. Consider $\kappa(\pi) = (3,2,1,2,2)$, this $\pi$ is balanced. $\kappa(r(\pi)) = (2,2,1,2,1)$, $\kappa(r^2(\pi)) = (1,2,1,2)$. We then have to reverse to apply $r$ again, $\kappa(r(2,1,2,1)) = (1,1,2)$. We have to reverse again, $\kappa(r(2,1,1)) = (1,1)$. The reduction now terminates and we conclude that $\pi$ is cyclic.

    If we consider an unbalanced permutation, say $\kappa(\pi) = (3,2,1,4,1)$, we are first in case $B$, reducing to $(1,1,2,1,4)$, we reverse to $(4,1,2,1,1)$ then again are in case $B$ and reduce to $(2,1,1,2,1)$, this is case $C$ which reduces to $(1,1,1,2)$, we reflect and get $(2,1,1,1)$, this is again case $C$ which reduces to $(1,1,1)$ since we hit a point where $a_1 = a_3$ we conclude that the original permutation is not cyclic. We could have concluded this originally because $6$ is clearly a fixed point for this permutation.

    Now let's consider a cyclic odd permutation $\kappa(\pi) = (2,4,3)$. First we need to reverse to get $(3,4,2)$. We are then in case $A$, and we reduce to $(2,4,1)$. This lands us in case $C$, which reduces us to $(1,4)$. We then have to reverse again to get $(4,1)$. This is case $B$, which reduces us to $(2,1)$. We are then in case $C$ and reduce to $(1)$.  The pictures of this reduction can be seen in Figure \ref{fig: reduce}.
\end{example}
\vspace{-.5em}
\begin{figure}[H]
\centering
\newcommand\myx[1][(0,0)]{\pic at #1 {myx};}
\tikzset{myx/.pic = {\draw
(-2.5mm,-2.5mm) -- (2.5mm,2.5mm)
(-2.5mm,2.5mm) -- (2.5mm,-2.5mm);}}
\newcommand\redmove[1]{$\vcenter{\hbox{$\;\xrightarrow{\ #1\ }\;$}}$}

\begin{tabular}{@{}c@{}c@{}c@{}c@{}c@{}c@{}c@{}}
\begin{tabular}{c}
$7\,8\,9\,3\,4\,5\,6\,1\,2$\\
\begin{tikzpicture}[scale=.38]
\draw[gray] (0,0) grid (9,9);
\foreach \x/\y in {1/7,2/8,3/9,4/3,5/4,6/5,7/6,8/1,9/2}
  \myx[(\x-.5,\y-.5)];
\draw[-,ultra thick,green!70!black]
 (0.5,0.5)--(0.5,6.5)--(6.5,6.5)--(6.5,4.5)--(4.5,4.5)--(4.5,3.5)
 --(3.5,3.5)--(3.5,2.5)--(2.5,2.5);
\draw[-,ultra thick,red!70!black]
 (2.5,2.5)--(2.5,8.5)--(8.5,8.5)--(8.5,1.5)--(1.5,1.5);
\draw[-,ultra thick,green!70!black]
 (1.5,1.5)--(1.5,7.5)--(7.5,7.5)--(7.5,0.5)--(0.5,0.5);
\draw[-,ultra thick,green!70!black,dashed]
 (2.5,2.5) to[bend left=35] (1.5,1.5);
\end{tikzpicture}\\
$(3,4,2)$
\end{tabular}
&
\redmove{A}
&
\begin{tabular}{c}
$6\,7\,2\,3\,4\,5\,1$\\
\begin{tikzpicture}[scale=.38]
\draw[gray] (0,0) grid (7,7);
\foreach \x/\y in {1/6,2/7,3/2,4/3,5/4,6/5,7/1}
  \myx[(\x-.5,\y-.5)];
\draw[-,ultra thick,green!70!black]
 (0.5,0.5)--(0.5,5.5)--(5.5,5.5)--(5.5,3.5)--(3.5,3.5)--(3.5,2.5)
 --(2.5,2.5)--(2.5,1.5)--(1.5,1.5);
\draw[-,ultra thick,red!70!black]
 (1.5,1.5)--(1.5,6.5)--(6.5,6.5)--(6.5,0.5)--(0.5,0.5);
\draw[-,ultra thick,green!70!black,dashed]
 (1.5,1.5) to[bend left=35] (0.5,0.5);
\end{tikzpicture}\\
$(2,4,1)$
\end{tabular}
&
\redmove{C}
&
\begin{tabular}{c}
$2\,3\,4\,5\,1$\\
\begin{tikzpicture}[scale=.38]
\draw[gray] (0,0) grid (5,5);
\foreach \x/\y in {1/2,2/3,3/4,4/5,5/1}
  \myx[(\x-.5,\y-.5)];
\draw[-,ultra thick,green!70!black]
 (0.5,0.5)--(0.5,1.5)--(1.5,1.5)--(1.5,2.5)--(2.5,2.5)--(2.5,3.5)--(3.5,3.5);
\draw[-,ultra thick,red!70!black]
 (3.5,3.5)--(3.5,4.5)--(4.5,4.5)--(4.5,0.5)--(0.5,0.5);
\draw[-,ultra thick,green!70!black,dashed]
 (3.5,3.5) to[bend left=25] (0.5,0.5);
\end{tikzpicture}\\
$(4,1)$
\end{tabular}
&
\redmove{B}
&
\begin{tabular}{c}
$2\,3\,1$\\
\begin{tikzpicture}[scale=.38]
\draw[gray] (0,0) grid (3,3);
\foreach \x/\y in {1/2,2/3,3/1}
  \myx[(\x-.5,\y-.5)];
\draw[-,ultra thick,green!70!black]
 (0.5,0.5)--(0.5,1.5)--(1.5,1.5);
\draw[-,ultra thick,red!70!black]
 (1.5,1.5)--(1.5,2.5)--(2.5,2.5)--(2.5,0.5)--(0.5,0.5);
\draw[-,ultra thick,green!70!black,dashed]
 (1.5,1.5) to[bend left=35] (0.5,0.5);
\end{tikzpicture}\\
$(2,1)$
\end{tabular}
\end{tabular}

\caption{The reduction of $\kappa^{-1}(2,4,3) = 894567123$. Each panel is drawn after reversing when necessary so that $a_1 > a_k$. This means the first and third panels show the inverse of the permutation produced by the previous step. Red marks the portion of the cycle removed by the indicated move, and the dashed green arc is the new edge that replaces it; after deleting, the surviving entries are relabelled. A final move $C$ applied to $(2,1)$ leaves $(1)$, so $\pi$ is cyclic.}
\label{fig: reduce}
\end{figure}
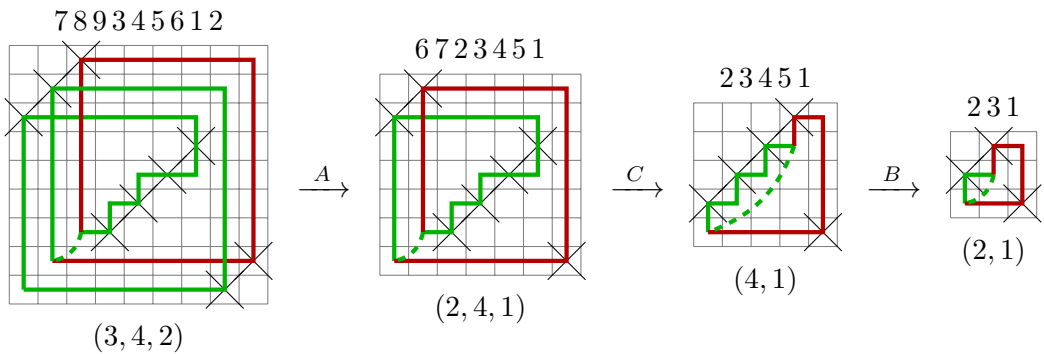

The previous example illustrated our new algorithm for verifying when a permutation $\pi \in \S_n(132,213)$ is cyclic,

\begin{algorithm} \label{alg: reduce}
    Given a permutation $\pi \in \S_n(132,213)$, repeatedly apply $r$, reversing when necessary to obtain $a_1 > a_k$, to $\pi$ until the corresponding composition has $a_1 = a_k$ or is just $a_1$. If this stopped state is either $(1,1)$ or $(1)$ the permutation is cyclic, if not it is not cyclic.
\end{algorithm}

Notice this recovers Huang's result without the need to balance the permutation. For sake of completion, we include this new proof,

\begin{corollary}\cite[Proposition 4.2]{H19}
    If $\pi \in \C_n(132,213)$, $\kappa(\pi)$ has exactly one odd part if $n$ is odd and exactly two odd parts if $n$ is even.    
\end{corollary}

\begin{proof}
By Theorem \ref{thm: reduction}, $\pi \in \C_n(132,213)$ if and only if $r(\pi)$ is as well. Applying Algorithm \ref{alg: reduce}, even permutations will terminate in $(1,1)$, while odd permutations terminate in $(1)$. Since our reduction always preserves the number of odd parts in the composition, even permutations have exactly two odd parts, while odd permutations have one.
\end{proof}

This new perspective has two main benefits:  it allows us to more easily work with all permutations $\pi \in \S_n(132,213)$, and it makes it easier to build up cyclic permutations from $(1,1)$ or $(1)$ via moves we describe now:

\begin{lemma} \label{lem: moves}
    Let $\pi \in \C_n(132,213)$ with $\kappa(\pi) = (a_1,a_2,\dots,a_k)$,
    \begin{enumerate}
        \item If the active boundary condition is $a_1|a_k$, subject to $a_1 > a_k \geq 1$, let $\Phi_A: a_1|a_k \to (2a_1-a_k)|a_1$. Then $\Phi_A(\pi) \in \C_{n+2(a_1-a_k)}(132,213)$.
        \item If the active boundary condition is $a_1,a_2|$ where $a_1,a_2 \geq 1$, let $\Phi_B: a_1,a_2| \to (a_1+2a_2)|a_2$. Then $\Phi_B(\pi) \in \C_{n+2a_2}(132,213)$.
        \item If the active boundary condition is $a_1|$, where $a_1 \geq 1$, let $\Phi_C: a_1| \to 2a_1|a_1$. Then $\Phi_C(\pi) \in \C_{n+2a_1}(132,213)$.
    \end{enumerate}
    In the case the permutation is balanced, the resulting permutation is balanced so long as when applying $\Phi_B$ we have $a_1 \not= n/2$. Furthermore, each of $\Phi_A, \Phi_B$ and $\Phi_C$ are injections.
\end{lemma}

\begin{proof}
    First we note that $\Phi_A$ and $\Phi_C$ clearly preserve being balanced because they always add the same amount to the beginning and end of the permutation. In $\Phi_B$, so long as both $a_1$ and $a_2$ are part of the first half of the balanced composition, $\Phi_B$ adds $a_2$ to the first and second part of the composition. If $a_2$ is part of the second composition, though, the image of $\Phi_B$ is not balanced. This only happens if $a_1 = n/2$. It remains to show that they preserve being cyclic.
    
    Let $\pi' = \Phi_A(\pi)$, so that the active boundary condition is $(2a_1-a_k)|a_1$. Notice then that $r(\pi') = \pi$ because $2a_1 > 2a_1-a_k > a_1$ since $a_1>a_k$ to start. By Theorem \ref{thm: reduction} since $\pi \in \C_n(132,213)$, so is $\pi'$. Similarly, let $\pi' = \Phi_B(\pi)$, so that the active boundary condition is $(a_1+2a_2)|a_2$. Again we find that $r(\pi') = \pi$ since in this case $a_1+2a_2 > 2a_2$. By Theorem \ref{thm: reduction} since $\pi \in \C_n(132,213)$, so is $\pi'$. Finally, let $\pi' = \Phi_C(\pi)$, so that the active boundary condition is $2a_1|a_1$. In this case we fall into the third case of our reduction, so $r(\pi') = \pi$ again. By Theorem \ref{thm: reduction} since $\pi \in \C_n(132,213)$, so is $\pi'$.

    Each of these maps is injective because the inverse map on their image is just the reduction of the resulting composition.
\end{proof}

For sake of completion, we also define $\Phi_R$ as the reversal involution, where $\Phi_R(a_1,\dots,a_k) = (a_k,\dots,a_1)$. 

\begin{example}
     Let's build up some cyclic permutations. If we start with $\omega = (1,1)$, then $\Phi_{C}(\omega) = (2,1,1)$ and $\Phi_B(\omega) = (3,1)$. Notice in the latter case, even though $\omega$ is balanced, since $a_2 = 1 = 2/2$, the result is not balanced. Notice also that $\Phi_A$ doesn't apply to $\omega$ at all, since it needs $a_1 > a_k$ and here $a_1 = a_k = 1$. So not all the moves are available at every stage.

    If we continue, since $\Phi_C(\omega) = (2,1,1)$, now we can apply $\Phi_B(\Phi_C(\omega)^R) = \Phi_B(1,1,2) = (3,2,1)$. Notice that to apply $\Phi_B$, we do not need the assumption that $a_1 > a_k$. This corresponds to the permutation $456231$. Notice if instead we wanted to generate $(3,1,2)$ we would apply $\Phi_A(2,1,1) = (3,1,2)$.

    Finally, if we apply $\Phi_A(3,1,2) = (4,1,3)$. If we consider $\Phi_C(3,1,4) = (6,1,4,3)$, which correspond to the permutation $9(10)(11)(12)(13)(14)84567123 \in \C_{14}(132,213)$. Notice we began this process with $(1,1)$ so all the permutations we found also had even size.
\end{example}

Notice that only $\Phi_C$ actually increases the length of the composition. Since this reduction and building up process are inverses to one another and deterministic (i.e. there is only one allowable reduction at any given time), we can conclude that every permutation $\pi \in \C_n(132,213)$ has a \emph{signature} which consists of the moves required to reduce it to $(1)$ or $(1,1)$. Explicitly,

\begin{definition}
    Given $\pi \in \C_n(132,213)$, let the \emph{signature} $s(\pi)$ be the reduction moves, listed from left to right, required to reduce $\pi$ to $(1)$, if $n$ is odd, and $(1,1)$ if $n$ is even. This is a word on the alphabet $\{A,B,C,R\}$.
\end{definition}

Whenever we write $s(\kappa(\pi))$, this is the same as $s(\pi)$. Our convention will be that $s(1) = s(1,1) = \varnothing$. Note that with the convention $\Phi_w = \Phi_{w_1} \circ \cdots \circ \Phi_{w_m}$, the rightmost letter of $w$ is applied first. Since $s(\pi)$ lists the reduction moves from left to right and each $\Phi$ inverts the corresponding reduction, we have $\Phi_{s(\pi)}(\omega) = \pi$, where $\omega = (1)$ or $(1,1)$ according to the parity of $n$. In particular, prepending a letter to a signature corresponds to applying one further $\Phi$ move at the end of the construction.

For example, $s(2,2,4,2,1) = CRCRCRBC$, and so $\Phi_{CRCRCRBC}(1) = (2,2,4,2,1)$. Notice the signature alone does not determine the permutation uniquely; it determines two permutations, one of each parity. Once the parity is known, though, the permutation is determined, since the reduction in Theorem \ref{thm: reduction} is deterministic. This gives us a way to connect the even and odd size permutations,

 \begin{theorem} \label{thm: bijs-odd-even}
 We have mutually inverse bijections
  \[
     \bigcup_{n \ge 1} \C_{2n-1}(132,213)\mathrel{
    \substack{
    \xrightarrow{\Phi_{s(\pi)}(1,1)}\\
    \xleftarrow[\hspace{.2em} \Phi_{s(\pi)}(1) \hspace{.2em}]{}
    }
    } \bigcup_{n \ge 1} \C_{2n}(132,213).
     \]
\end{theorem}

 \begin{proof}
    Given some $\pi \in \C_{2m-1}(132,213)$, we first take its signature $s(\pi)$. This means that $\Phi_{s(\pi)}(1) = \pi$. We instead apply $\Phi_{s(\pi)}$ to $(1,1)$. The resulting permutation has even length by Lemma \ref{lem: moves} because each $\Phi$ move increases the size by an even amount, preserving the parity. Since each of the $\Phi$ are injections, their composition is an injection. We can do the exact same thing for any $\pi' \in \C_{2j}(132,213)$ giving us an injection in the other direction, from which we conclude each map is a bijection. Notice that these maps are mutually inverse bijections precisely because for $\pi \in \C_{2m-1}(132,213)$, $s(\Phi_{s(\pi)}(1,1)) = s(\pi)$. So we have,
    \[
    \Phi_{s(\Phi_{s(\pi)}(1,1))}(1) = \pi,
    \]
    likewise for $\pi' \in \C_{2j}(132,213)$, we have $\Phi_{s(\Phi_{s(\pi')}(1))}(1,1) = \pi'$.
 \end{proof}

  Despite the fact that every permutation $\pi \in \C_n(132,213)$ has a unique signature up to knowing the parity of the size of the permutation, it seems difficult to understand how the length of the signature is related to the size of the permutation. That is, the general difficulty in counting $\C_n(132,213)$ exactly via signatures is that the growth we see at each point is dependent on the active boundary. We can find the length of the composition if we know its signature and parity,

  \begin{proposition}
      For $n \ge 1$ if $\pi \in \C_n(132,213)$, let $[C](s(\pi))$ denote the number of $C$'s in the signature,
      \[
      \ell(\kappa(\pi)) = 
      \begin{cases}
          [C](s(\pi)) + 1 & \text{if} \; n \; \text{is odd}\\
          [C](s(\pi)) + 2 & \text{if} \; n \; \text{is even}
      \end{cases}
      \]
  \end{proposition}

  \begin{proof}
      Only $\Phi_C$ increases the length of the corresponding composition. Since odd sized compositions reduce to $(1)$, they begin with length $1$ and every application of $\Phi_C$ increases their length by $1$. Even sized compositions reduce to $(1,1)$, so they begin with size $2$ and each application of $\Phi_C$ increases their length by $1$. This precisely recovers the stated result.
  \end{proof}

\begin{example}
  As an example, again consider $(2,2,4,2,1)$, this has length $5$ and $n = 11$, and $s(2,2,4,2,1) = CRCRCRBC$. There are four $C$'s in the signature, so since $n$ is odd, the corresponding composition will have length $5$. 
\end{example}
  
  We explore the connection between the size of a permutation $\pi \in \C_n(132,213)$ and its length a bit more in Section \ref{sec: restricted-parts}. We will now use these tools to get a lower bound on the growth rate.

\section{Lower Bound} \label{sec: lower-bound}

We now use the tools we just developed to prove the first non-trivial lower bound on the growth rate of $\C_n(132,213)$. We let $\A_n$ denote the set of $\pi \in \C_{n}(132,213)$ where $\kappa(\pi)$ has active boundary $|1$. This just means that the final entry of the composition is $1$. The ultimate goal of this section is to prove the following lower bound,

\begin{namedthm*}{Theorem \ref{thm: lowerBound}}
    $\liminf_{n \to \infty} |\C_{n}(132,213)|^{1/n} \geq  1.23693$
\end{namedthm*}

We will do this by finding a lower bound on the growth rate of the $\A_n$ sets by recursively injecting smaller $\A_m$ into $\A_n$. We break the proof of this result into some lemmas for ease of reading. First, some more notation. Let $\mathcal{Z}_{2k,n}$ denote the set of $\pi \in \C_{n}(132,213)$ where $\kappa(\pi)$ has active boundary condition $2k|1$ for all $k \geq 1$. This means the first entry of the composition is $2k$ and the final entry is $1$. Note that for fixed $n$ all of the $\cZ_{2k,n}$ are mutually disjoint because compositions with different active boundary conditions are distinct. We can therefore conclude that
\[
|\A_n| \ge \sum_{k=1}^{\lfloor n/2 \rfloor} |\cZ_{2k,n}|.
\]

Our first result connects certain $\cZ$ sets to the $\A$ sets,

\begin{lemma} \label{lem: A-to-Z2}
    For $m \ge 1$, $|\cZ_{2,m+2}| = |\A_{m}|$.
\end{lemma}

\begin{proof}
    If $m = 1$, the result is trivially true because $|\A_1| = 1$ and $|\cZ_{2,3}| = 1$ as well. Now assume $m \ge 2$, let $\pi \in \A_{m}$, by definition $\kappa(\pi)$ has active boundary $a_1|1$. If we apply $\Phi_C$ to the reverse, $\Phi_C(\Phi_R(\pi))$, this has active boundary $2|1$ and its size has increased by $2$. So $\Phi_C(\Phi_R(\pi)) \in \cZ_{2,m+2}$. Since both $\Phi_C$ and $\Phi_R$ are injections, so is their composition. Furthermore, since every $\pi \in \cZ_{2,m+2}$ has $a_1 = 2$ and $a_k = 1$, the first case in its reduction is case $C$ and then $R$ which is precisely the inverse of our injection.
\end{proof}

Now we want to relate any $\cZ_{2k,n}$ back to $\cZ_{2,j}$ for some $j$,

\begin{lemma}\label{lem: INJ-1}
    For $k \ge 2$ and $m > 4k-2$, the map $\Phi_{B}\Phi_R\Phi_A^{k-2}\Phi_C$ is a well defined injection from $\cZ_{2,m-(4k-2)}$ into $\cZ_{2k,m}$.
\end{lemma}

\begin{proof}
    Given $\pi \in \cZ_{2,m-4k+2}$, we claim $\Phi_B\Phi_R\Phi_A^{k-2}\Phi_C(\pi) \in \cZ_{2k,m}$. Let's track what happens to $\pi$ and how much $m$ changes. $\kappa(\pi) = (2,a_2,\dots,a_{j-1},1)$, after applying $\Phi_C$ we get $(4,a_2,\dots,a_{j-1},1,2)$ at a cost of $4$. Every time we apply $\Phi_A$ now it increases both the first and last entry by $2$ and costs $4$, so if we apply it $k-2$ times we end up with $(4+2k-4,a_2,\dots,a_{j-1},1,2k-2)$ at a cost of $4(k-2)$. Now if we reverse we get $(2k-2,1,a_{j-1},\dots,a_2,2k)$, if we apply $\Phi_B$ we end up with $(2k,a_{j-1},\dots,a_2,2k,1)$ at a cost of $2$. This has first entry $2k$, final entry $1$ and we have increased the size of the composition by $4 + 4(k-2) + 2 = 4k-2$. Since by Lemma \ref{lem: moves} each of the $\Phi$ maps is an injection, so is their composition.
\end{proof}

From this we can conclude that $|\cZ_{2k,m}| \ge |\cZ_{2,m-(4k-2)}|$. This would give us a lower bound (see Remark \ref{remark: truncation}), but we can do better.

\begin{lemma}\label{lem: INJ-2}
    For $k \ge 2$ and $m > 4k-2$, the map $\Phi_B \Phi_R \Phi_C$ is an injection from $\cZ_{2k-2,m-(4k-2)}$ into $\cZ_{2k,m}$.
\end{lemma}

\begin{proof}
    Given $\pi \in \cZ_{2k-2,m-(4k-2)}$ by definition its boundary is $2k-2|1$. If we apply $\Phi_C$ we get a permutation with boundary $4k-4|1,2k-2$ whose size has increased by $4k-4$. Now reverse to get $2k-2,1|4k-4$. If we apply $\Phi_B$ we get a permutation with boundary $2k|4k-4,1$ with size increased by $2$, which is now in $\cZ_{2k,m}$. In all, we found $\Phi_B\Phi_R\Phi_C(\pi) \in \cZ_{2k,m}$. Since each map is an injection, so is their composition.
\end{proof}

Notice these two maps agree when $k = 2$. But for $k \ge 3$, they give different images. We will make this explicit now,

\begin{lemma} \label{lem: disjoint-injs}
    For $k \ge 3$, the images of the two injections from Lemmas \ref{lem: INJ-1} and \ref{lem: INJ-2} are disjoint.
\end{lemma}

\begin{proof}
    Let $\pi = \Phi_B\Phi_R\Phi_A^{k-2}\Phi_C(\tau)$ for $\tau \in \cZ_{2,m-(4k-2)}$ be in the image of the map from Lemma \ref{lem: INJ-1}, and let $\pi' = \Phi_B\Phi_R\Phi_C(\tau')$ for $\tau' \in \cZ_{2k-2,m-(4k-2)}$ be in the image of the map from Lemma \ref{lem: INJ-2}. Since the reduction process is deterministic and inverse to our $\Phi$ maps, we can read off
    \[
    s(\pi) = B\,R\,A^{k-2}\,C\,s(\tau) \qquad \text{and} \qquad s(\pi') = B\,R\,C\,s(\tau').
    \]
    For $k \ge 3$ we have $k-2 \ge 1$, so $s(\pi)$ begins with $BRA$ while $s(\pi')$ begins with $BRC$. Permutations with different signatures must be unique by Theorem \ref{thm: reduction}.
\end{proof}

Notice the case $k = 2$ is excluded from this Lemma because there the $\Phi_A^{k-2}$ is empty and the maps agree. We will only count this case once through Lemma \ref{lem: INJ-1}. We are almost ready to prove the main result. For ease of notation, let $\alpha_1(n) = |\A_{n-2}|$, $\alpha_2(n) = |\A_{n-8}|$ and for $k \ge 3$ let
    \[
    \alpha_k(n) = |\A_{n-4k}| + \alpha_{k-1}(n-(4k-2)),
    \]
where $|\A_1| = 1$, $|\A_2| = 1$ and by convention $|\A_m| = 0$ for $m \le 0$.

\begin{theorem} \label{thm: lowerBound}
    $\liminf_{n \to \infty} |\C_n(132,213)|^{1/n} \ge 1.23693$.
\end{theorem}

\begin{proof}
    We will show that for every $n \ge 1$ and $k \ge 1$, we have $|\cZ_{2k,n}| \ge \alpha_k(n)$ from which we conclude
    \[
    |\A_n| \ge \sum_{k \ge 1} \alpha_k(n).
    \]
    We prove $|\cZ_{2k,n}| \ge \alpha_k(n)$ by induction on $k$. The base case $k=1$ is implied by Lemma \ref{lem: A-to-Z2}. For $k=2$, Lemma \ref{lem: INJ-1} gives $|\cZ_{4,n}| \ge |\cZ_{2,n-6}| = |\A_{n-8}| = \alpha_2(n)$, using Lemma \ref{lem: A-to-Z2} for the middle equality. Now let $k \ge 3$ and assume the claim for $k-1$. By Lemmas \ref{lem: INJ-1} and \ref{lem: INJ-2} we have
    \[
    |\cZ_{2k,n}| \ge |\cZ_{2,n-4k+2}| + |\cZ_{2k-2,n-(4k-2)}| \ge |\A_{n-4k}| + |\cZ_{2k-2,n-(4k-2)}|,
    \]
    since the two injections are disjoint when $k \ge 3$ by Lemma \ref{lem: disjoint-injs}. If we apply the inductive hypothesis we conclude that $|\cZ_{2k-2,n-(4k-2)}| \ge \alpha_{k-1}(n-(4k-2))$. Substituting this in we conclude $|\cZ_{2k,n}| \ge \alpha_k(n)$. The result follows by induction.

    Since $|\A_n| \ge \sum_{k=1}^{\lfloor n/2 \rfloor} |\cZ_{2k,n}|$ we conclude $|\A_n| \ge \sum_{k \ge 1} \alpha_k(n).$ This implies 
    \[
    \liminf_{n \to \infty} |\C_n(132,213)|^{1/n} \ge \liminf_{n \to \infty} |\A_n|^{1/n} \ge 1.23693
    \] 
    as claimed.
\end{proof}

\begin{remark}
    To find this bound, one expands the recursion, $|\A_n| \ge |\A_{n-2}| + \sum_{d \in D} m_d |\A_{n-d}|$ for an explicit set of shifts $D$ with multiplicities $m_d$. The bound then follows by considering the roots of $1 - x^2 - \sum_{d \in D} m_d x^d$, truncated at $d \le 200$; we use Sage \cite{Sage} to approximate the dominant root.
\end{remark}

\begin{remark} \label{remark: truncation}
   We note that without the recursion, just using Lemma \ref{lem: INJ-1}, one can derive a lower bound of approximately $1.2143$ from $|\A_n| \ge |\A_{n-2}| + \sum_{k \ge 2} |\A_{n-4k}|$. The recursion improves the lower bound so we included it.
\end{remark}

In \cite{H19}, Huang experimentally estimated that the actual growth rate is somewhere around $2^{.38}\approx 1.301$ so this is fairly close and we only consider the compositions that start or end with $1$. We still miss some elements in our injections into $\cZ_{2k,n}$. One way to potentially improve this bound is to add terms to the recursion by finding injections into other parts of $\cZ_{2k,n}$.

\subsection{Another approach}

We have another approach to this lower bound that we think could be improved, but currently provides a worse lower bound than Theorem \ref{thm: lowerBound}. We include it in the hopes that it could lead to a better lower bound eventually. We let $\cB_n$ denote the set of $\pi \in \C_n(132,213)$ where $\kappa(\pi)$ has active boundary $|1,a_k$, meaning the second to last entry in the composition is $1$.

\begin{lemma} \label{lem: AtoAB}
    For $n \geq 1$ we have $|\A_n| = |\A_{n-2}| + |\cB_{n-2}|$.
\end{lemma}

\begin{proof}
    Let $\pi \in \A_n$, there are two cases to consider based on how $\kappa(\pi)$ begins. By Proposition \ref{prop: unequal-ends} we know $a_1 \not= 1$, so it either has active boundary $2|1$ or $a_1|1$ with $a_1 > 2$. In the first case, if we wind down we get $r(\pi)$ with active boundary $1|$ of size $n-2$. These are in bijection with $\A_{n-2}$. In the second case, $r(\pi)$ has boundary $a_1-2,1|$. The reversal of this is an element of $\cB_{n-2}$.

    This shows $|\A_{n-2}| + |\cB_{n-2}| \geq |\A_n|$, but both of these winding down moves are easily reversible. To inject $\cB_{n-2}$ into $\A_n$ we just reverse the winding down move, meaning we apply $\Phi_B\Phi_R$. To inject $\A_{n-2}$ into $\A_n$ we just apply $\Phi_C\Phi_R$. Their images are disjoint because every element in the image of $\A_{n-2}$ has signature beginning with $CR$ while every element in the image of $\cB_{n-2}$ has signature beginning with $BR$. Permutations with different signatures must be unique by Theorem \ref{thm: reduction}. This proves the desired equality.
\end{proof}

\begin{remark}
    We note that if $n$ is even, it is possible that two elements of $\A_n$ wind down to the same composition. For example, $(3,2,1)$ and $(2,1,2,1)$ both wind down to $(1,1,2)$, whose reversal $(2,1,1)$ then sits in both $\cB_{n-2}$ and $\A_{n-2}$. This occurs because the reduction process is not injective, just the constructive moves are. Notice the two preimages are still distinguished by their signatures: $s(3,2,1) = BRC$ while $s(2,1,2,1) = CRC$, which is exactly the $BR$ versus $CR$ distinction from the proof.
\end{remark}

\begin{lemma} \label{lem: BbelowA}
    For $n \geq 7$, we have $|\cB_{n}| \geq |\A_{n-6}|$.
\end{lemma}

\begin{proof}
    Given $\pi \in \A_{n-6}$, we know $\pi$ has active boundary $a_1|1$. If we apply $\Phi_C\Phi_C\Phi_R$, the resulting permutation has active boundary $4|a_1,1,2$ and its size has increased by $6$. This is in $\cB_n$ and since it is a composition of injections, we conclude $|\A_{n-6}| \le |\cB_n|$.
\end{proof}

\begin{proposition}
    For $n \geq 7$, $|\A_n| \ge |\A_{n-2}| + |\A_{n-8}|$, which implies $\liminf_{n \to \infty} |\C_{n}(132,213)|^{1/n} \geq 1.1748$.
\end{proposition}

\begin{proof}
    Combining Lemmas \ref{lem: AtoAB} and \ref{lem: BbelowA} yields the first claim. The second claim then follows immediately by considering the roots of $1-x^2-x^8$.
\end{proof}

\begin{conjecture} \label{conj: A-B-conj}
    For $n \geq 3$, $|\cB_n| \ge |\A_{n-2}|$.
\end{conjecture}

Experimentally this seems true, but one would have to construct an injection that does not rely on the reduction operators and their inverse moves from Lemma \ref{lem: moves}. This would give a lower bound of $1.272$ which is very close to the conjectured growth rate in Huang. It also experimentally seems true that $\A_n$ and $\C_n(132,213)$ have the same growth rate. Since $\A_n \subseteq \C_n(132,213)$ one direction is immediate, the content is the other. We pose this as a question instead of a conjecture,

\begin{question} \label{Q: C-to-A-growth}
    Do $\A_n$ and $\C_n(132,213)$ have the same exponential growth rate? More strongly, is $\liminf_{n \to \infty} |\A_n|/|\C_n(132,213)| > 0$?
\end{question}

\section{Restricted Composition Length} \label{sec: restricted-parts}

In this section we explore what our reduction algorithm tells us about permutations whose compositions have restricted length, i.e. the corresponding permutation has a restricted number of layers. Recall that a permutation is Grassmannian if it has at most one descent.

\begin{theorem} \label{thm: two-parts}
    For $n \geq 2$, if $\kappa(\pi) = (a_1,a_2)$, then $\pi \in \C_n(132,213)$ if and only if $\gcd(a_1,a_2) = 1$. As a result, the number of Grassmannian $\pi \in \C_n(132,213)$ is $\phi(n)$, where $\phi$ is Euler's totient function.
\end{theorem}

\begin{proof}
    The Grassmannian permutations in $\S_n(132,213)$ are precisely those whose composition has exactly two parts; the one-part composition gives the identity, which is not cyclic for $n \ge 2$. So we must count the pairs $(a_1,a_2)$ with $a_1 + a_2 = n$ for which $\kappa^{-1}(a_1,a_2)$ is cyclic.

    Suppose $\kappa(\pi) = (a_1,a_2)$ with $a_1 > a_2$, and set $g = \gcd(a_1,a_2)$. As long as $a_1 \neq 2a_2$, the reduction stays within two-part compositions. Indeed, move $A$ replaces $(a_1,a_2)$ by $(a_2, 2a_2-a_1)$ and move $B$ replaces it by $(a_1-2a_2,a_2)$, and in both cases the gcd of the two parts is unchanged since 
    \[
    \gcd(a_1,a_2) = \gcd(2a_2-a_1,a_2) = \gcd(a_1-2a_2,a_2).
    \] 
    Move $R$ also preserves the gcd. Since each of $A$ and $B$ strictly decreases $a_1+a_2$, the reduction reaches, after finitely many steps, either $(b,b)$ with $b = g$, or the case $b_1 = 2b_2$, at which point move $C$ applies and produces $(b_2)$.

    If $g > 1$, then no reduced composition can be $(1,1)$ or $(1)$, since $g$ divides both parts throughout and divides the surviving part after move $C$. Hence $\pi$ is not cyclic.

    On the other hand, suppose $g = 1$. The two parts can never become equal, since equality would force $g > 1$, unless both equal $1$. Thus the reduction terminates either at $(1,1)$, or by applying move $C$ to some $(2b_2,b_2)$ with $\gcd(2b_2,b_2) = b_2 = g = 1$, producing $(1)$. In either case Algorithm \ref{alg: reduce} shows that $\pi$ is cyclic.

    Therefore the Grassmannian elements of $\C_n(132,213)$ are exactly the $\kappa^{-1}(k,n-k)$ with $1 \le k \le n-1$ and $\gcd(k,n-k) = \gcd(k,n) = 1$, of which there are $\phi(n)$.
\end{proof}

We now consider the case where the composition has three parts, it relies on the previous result,

\begin{theorem}
    For $n \ge 1$, if $\kappa(\pi) = (a_1,a_2,a_3)$, then $\pi \in \C_n(132,213)$ if and only if $\gcd(a_1+a_2,a_2+a_3) = 1$.
\end{theorem}

\begin{proof}
   We will consider each possible reduction case. First, the two cases that preserve the number of components in the composition. We will argue that they preserve the pairwise sum gcd. If we fall into the case where $2a_3 > a_1 > a_3$, then when we reduce we get $(a_3,a_2,2a_3-a_1)$. In this case, 
   \[
   \gcd(a_2+a_3,a_2 + 2a_3-a_1) = \gcd(a_2+a_3,2(a_2+a_3)-(a_1+a_2))= \gcd(a_2+a_3, a_1+a_2)
   \]
   so the gcd of the pairwise sums is preserved. If we fall into the case where $a_1 > 2a_3$, we reduce to get $(a_1-2a_3,a_3,a_2)$. The gcd of the pairwise sums is then
   \[
   \gcd((a_1-2a_3)+a_3,a_3+a_2)=\gcd((a_1+a_2)-(a_2+a_3),a_2+a_3) = \gcd(a_1+a_2, a_2+a_3)
   \]
   so again we find the pairwise sum gcd is preserved. Finally, if we have $a_1 = 2a_3$, our composition reduces to $(a_3,a_2)$. In which case, by Theorem \ref{thm: two-parts}, the composition is cyclic if and only if $\gcd(a_3,a_2) = 1$. But the original gcd in this case is $\gcd(2a_3+a_2,a_2+a_3) = 1$. This is the same as $\gcd(a_3,a_2+a_3) = \gcd(a_3,a_2) = 1$. Since the first two reductions preserve the pairwise sum gcd and the final property is cyclic so long as $\gcd(a_1+a_2,a_2+a_3) = 1$, the result follows.
\end{proof}

It appears that when the composition has more than four components, this type of result becomes far more difficult. It would still be interesting to prove a version of these results for other compositions with restricted lengths.

\section{Balanced Results} \label{sec: balanced}

In this section, we prove some new structural results about the permutations in $\C_n(132,213)$ whose compositions are balanced. Our hope is that these tools could help prove \cite[Conjecture 5.1]{H19}. Let $\C^B = \bigcup_{n \geq 1} \C_{2n}^B(132,213)$.

\begin{lemma} \label{lem: move-unbalance}
    If $\pi \in \C^B$ with $\kappa(\pi) = (a_1,a_2,\dots,a_k)$
    \begin{itemize}
        \item If $a_1 > a_k$, then $\Phi_A(\pi) \in \C^B$
        \item If $a_1 \not= n/2$ then $\Phi_B(\pi) \in \C^B$
        \item $\Phi_C(\pi) \in \C^B$
        \item $\Phi_R(\pi) \in \C^B$
        \item If $\pi \not\in \C^B$ then $\Phi_A(\pi)$, $\Phi_B(\pi)$, $\Phi_C(\pi)$ and $\Phi_R(\pi)$ are not in $\C^B$.
    \end{itemize}
\end{lemma}

\begin{proof}
    The first four bullets all follow from Lemma \ref{lem: moves}. We wanted to restate them here for completeness. The main content is the fifth bullet point.
    
    If $\pi \not\in \C^B$, let's suppose $\kappa(\pi) = (a_1,a_2,\dots,a_k)$ with $a_1+a_2+ \cdots + a_k = 2n$ and without loss that $a_1 > a_k$. Then $\Phi_A(\pi) = (2a_1 - a_k, a_2,\dots,a_{k-1}, a_1)$ and has size $2n + 2(a_1-a_k)$. If we have $2a_1-a_k = n + a_1-a_k$ this means $a_1 = n$ which implies that $\pi$ is balanced. If instead some prefix sum equals $n + a_1 - a_k$ this means
    \[
    2a_1 - a_k + \sum_{i=2}^j a_i = n + a_1 - a_k.
    \]
    If we subtract $(a_1-a_k)$ from both sides we conclude that $a_1 + a_2 + \cdots + a_j = n$ again contradicting that $\pi \not\in \C^B$. Finally, we have to consider the case that $a_1 = n + a_1 - a_k$, this implies that $n - a_k = 0$ which means $n = a_k$ which is not possible since $a_1 > a_k$.

    Next, suppose $\Phi_C(\pi) = (2a_1,a_2,\dots,a_k,a_1) \in \C^B$. A similar analysis shows that if $\Phi_C(\pi)$ is balanced, then so is $\pi$.
    
    Finally, we will argue that $\Phi_B(\pi) \not\in \C^B$. There are a few cases to consider. First,  since $\pi$ is not balanced, there is never a point where $a_1 + \cdots + a_j = n$. Suppose for sake of contradiction that $\Phi_B(\pi) = (a_1+2a_2,a_3,\dots,a_k,a_2) \in \C_{2n+2a_2}^B$.  There are a few ways for this to be balanced. First, we could have $a_1+2a_2 = n + a_2$, however if this were the case it implies that $a_1 + a_2 = n$ which implies that $\pi \in \C_{2n}^B$ which is a contradiction. If the first element is not equal to $n + a_2$, we could have some initial partial sum $(a_1+2a_2) + \sum_{i=3}^j a_i = n + a_2$ for $3 \leq j \le k$. However, subtracting $a_2$ this means $a_1 + a_2 + a_3 + \cdots + a_j = n$ which implies again that $\pi$ is balanced. Since $a_1 + 2a_2 > a_2$, it can never be the case that $a_2 = n + a_2$ unless $n = 0$. This means $\Phi_B(\pi)$ can never be balanced if $\pi$ was not balanced.
\end{proof}

The final item in the lemma says that when constructing a permutation, if we ever create an unbalanced permutation, it is impossible to re-balance. Another way to read this is that if any of our moves produce a balanced permutation, the original permutation had to be balanced. This is almost a converse to the first four bullet points outside of move $\Phi_B$. This gives us a rigorous way to check when a signature exits being balanced and once it does, it can never return. We discuss some ways to leverage this in the future works section. For now, we prove a couple other structural results about balanced compositions.

\begin{theorem} \label{thm: product}
    For $n \geq 1$, let $\pi \in \S_{2n}(132,213)$, so that $\pi = \sigma_1 \ominus \sigma_2$ for $\sigma_1,\sigma_2 \in \S_n(132,213)$. Then $\pi \in \C_{2n}^B(132,213)$ if and only if $\sigma_2 \sigma_1$ is cyclic.
\end{theorem}

\begin{proof}
    By definition of the skew sum, $\pi(i) = \sigma_1(i) + n$ for $1 \le i \le n$ and $\pi(n+j) = \sigma_2(j)$ for $1 \le j \le n$. Notice this means $\pi$ interchanges $L = \{1,\dots,n\}$ and $E = \{n+1,\dots,2n\}$, and that $n$ is a partial sum of $\kappa(\pi)$, so every permutation of this form is balanced.

    Since $\pi$ swaps $L$ and $E$, we know $\pi^2$ preserves each of them. This means for $i \in L$,
    \[
    \pi^2(i) = \pi(\sigma_1(i)+n) = \sigma_2(\sigma_1(i)),
    \]
    so $\pi^2|_L = \sigma_2\sigma_1$.

    Now, suppose $\sigma_2\sigma_1$ is cyclic. This means $\pi^2$ is transitive on $L$, so the cycle of $\pi$ containing $1$ contains every element of $L$. But since $\pi(L) =E$, if we apply $\pi$ again, this means the cycle of $\pi$ containing $1$ also contains every element of $E$ as well. So $\pi$ must be a full $2n$-cycle.

    Conversely, suppose $\pi$ is cyclic. Because $\pi$ interchanges $L$ and $E$, no cycle of $\pi^2$ can contain elements from both $L$ and $E$. Since $\pi$ is cyclic and has even length, $\pi^2$ has exactly two cycles, each of length $n$. This forces one cycle to be all of the elements of $L$ and the other to be all of the elements of $E$. In particular $\pi^2|_L = \sigma_2\sigma_1$ is cyclic as desired.
\end{proof}

This means that finding balanced cyclic $132,213$-avoiding permutations is the same as asking which products of $132$ and $213$ avoiding permutations are cyclic. It makes sense, then, that this is a difficult question because we are asking for an algebraic property (multiplication) to interact with avoidance.

\begin{corollary}
    For $n \geq 1$, $|\C_n(132,213)|$ is equal to the number of $\pi \in \C_{2n}^B(132,213)$ with $\kappa(\pi) = (n,a_2,\dots,a_k)$.
\end{corollary}

\begin{proof}
    The condition $\kappa(\pi) = (n,a_2,\dots,a_k)$ says exactly that $\sigma_1 = {\rm id}_n$. By Theorem \ref{thm: product}, $\pi$ is cyclic if and only if $\sigma_2\sigma_1 = \sigma_2$ is, so this exactly corresponds to the case where $\sigma_2$ must be cyclic.
\end{proof}

Notice that these are precisely the elements of $\C_{2n}^B(132,213)$ that $\Phi_B$ sends to unbalanced compositions. 

\begin{example}
    Let $\pi = \kappa^{-1}(2,2,1,3) = 78564123 \in \S_8(132,213)$. Here $n = 4$, $\sigma_1 = 3412$ and $\sigma_2 = 4123$. Notice that in this case $\sigma_1 = (1,3)(2,4)$ is not cyclic. But the product $\sigma_2\sigma_1 = 2341$, which is the cyclic permutation $(1,2,3,4)$. So Theorem \ref{thm: product} guarantees that $\pi = \sigma_1 \ominus \sigma_2 \in \C_8^B(132,213)$, and indeed $\pi = (1,7,2,8,3,5,4,6)$.
\end{example}

\section{Future Work} \label{sec: future-work}

Despite our best efforts, we could not use any of the new tools to improve the upper bound on the growth rate. A few possible approaches include:
\begin{itemize}
    \item Trying to better understand the size of a composition corresponding to a given signature. Each move except $R$ increases the size of the composition by at least $2$. We believe that outside of special cases, the size is often increased by more than $2$. Also, not every word on $\{A,B,C,R\}$ is a valid signature: $R$ can never be applied twice, and $R$ will never be applied after $A$ because $A$ preserves the property that $a_1 > a_k$. This means the number of signatures of length $m$ grows slower than $4^m$.

    It is possible that with further study, by getting a bound on the number of signatures of length $m$ one could produce an upper bound on the number of $\pi \in \C_n(132,213)$.

    \item If Question \ref{Q: C-to-A-growth} is answered positively, one could find an upper bound on $\C_n(132,213)$ by bounding $\cB_n$ from above in Lemma \ref{lem: AtoAB}.

    \item If one could get a better understanding of which pairs $\sigma_1,\sigma_2 \in \S_n(132,213)$ have $\sigma_2\sigma_1$ cyclic, Theorem \ref{thm: product} could give a lower bound on the number of balanced permutations, which could lead to a general lower bound.
\end{itemize}

Using some of the results in Section \ref{sec: balanced}, it might be possible to approach Huang's conjecture

\begin{conjecture}{\cite[Conjecture 5.1]{H19}}
    For even $n$,
    \[
    \frac{|\C_n^B(132,213)|}{|\C_n(132,213)|} = \Omega(1).
    \]
\end{conjecture}

One approach could be to first classify the ratio of signatures of a given length that produce balanced vs unbalanced permutations using Lemma \ref{lem: move-unbalance}, then try to use this to prove the conjecture. This would again require some understanding of how the length of a signature connects to the size of the resulting permutation. We conclude by collecting the questions and conjectures posed throughout the paper.

\begin{namedthm*}{Conjecture \ref{conj: A-B-conj}}
    For $n \geq 3$, $|\cB_n| \ge |\A_{n-2}|$.
\end{namedthm*}

This in combination with Lemma \ref{lem: AtoAB} would give an improved lower bound of $1.272$. We had one other question which would be interesting to answer,

\begin{namedthm*}{Question \ref{Q: C-to-A-growth}}
    Do $\A_n$ and $\C_n(132,213)$ have the same exponential growth rate? More strongly, is
    $\liminf_{n \to \infty} |\A_n|/|\C_n(132,213)| > 0$?
\end{namedthm*}

One could then focus on the $\A_n$ sets instead of all of $\C_n(132,213)$ for questions on growth rates. Of course the question of exactly enumerating $\C_n(132,213)$ is still open. Our hope is that some of these tools could help answer this question.

\bibliographystyle{amsplain}

\noindent {\tiny \emph{The views expressed in this paper are those of the authors and do not reflect the official policy or position of the U.S. Naval Academy, Department of the Navy, the Department of Defense, or the U.S. Government.} \par}

\end{document}